\documentclass[11pt,letterpaper]{amsart}

\usepackage{tikz, tikz-3dplot}
\usepackage{amsmath,amsthm,amssymb,graphics}
\usepackage[utf8]{inputenc}
\usepackage[english]{babel}
\usepackage[colorlinks=true, allcolors=blue]{hyperref}
\usepackage{slashbox}
\usepackage[title]{appendix}
\usepackage{multirow}
\usepackage{setspace}
\usepackage{comment}

\theoremstyle{plain}

\newtheorem*{ack*}{Acknowledgment}

\newtheorem{te}{Theorem}[section]
\newtheorem{coro}[te]{Corollary}

\newtheorem{defn}[te]{Definition}
\newtheorem{lem}[te]{Lemma}

\newcommand{\dsum}{\displaystyle\sum}
\newcommand{\dint}{\displaystyle\int}

\newcommand{\dprod}{\displaystyle\prod}
\newcommand{\nocontentsline}[3]{}

\def\x{{\boldsymbol x}}
\def\y{{\bf y}}

\def\b{{\boldsymbol b}}

\def\g{{\boldsymbol g}}
\def\0{{\bf 0}}
\def\a{{\boldsymbol a}}
\def\b{{\boldsymbol b}}

\def\c{{\boldsymbol c}}
\def\m{{\boldsymbol m}}
\def\h{{\boldsymbol h}}
\def\y{{\boldsymbol y}}

\def\r{{\boldsymbol r}}

\def\R{{\mathbb R}}
\def\Q{{\mathbb Q}}

\def\N{{\mathbb N}}

\def\Z{{\mathbb Z}}
\def\P{{\mathbb P}}

\numberwithin{equation}{section}

\newcommand{\nmod}[1]{\,\,\text{\rm mod}\,\,#1}

\def\bfw{{\mathbf w}}
\def\bfx{{\mathbf x}}
\def\bfy{{\mathbf y}}
\def\bfz{{\mathbf z}}

\def\N{{\mathbb N}}  \def\P{{\mathbb P}}
\def\R{{\mathbb R}}
\def\Z{{\mathbb Z}}\def\Q{{\mathbb Q}}

\def\eps{\varepsilon}

\def\le{\leqslant} \def\ge{\geqslant}

\begin{document}

\author{Daniel Flores Galiote and Kiseok Yeon}

\email{flore205@purdue.edu}
\address{Department of Mathematics, Purdue University, 150 N. University Street, West Lafayette, IN 47907-2067}
\email{kyeon@ucdavis.edu}
\address{Department of Mathematics, University of California, Davis, One Shields Avenue, Davis, CA 95616, USA}
\subjclass[2020]{11E76,14G12}
\keywords{Homogenous polynomials, Hasse principle}

\title[Daniel Flores Galiote and Kiseok Yeon]{The Hasse principle for random homogeneous polynomials in thin sets}

\maketitle

\begin{center}
    \textit{In honor of Trevor Wooley's 60th birthday}
\end{center}

\begin{abstract}
Let $d$ and $n$ be natural numbers. Let $\nu_{d,n}: \mathbb{R}^n\rightarrow \mathbb{R}^{N}$ denote the Veronese embedding with $N=N_{n,d}:=\binom{n+d-1}{d}$, defined by listing all the monomials of degree $d$ in $n$ variables using the lexicographical ordering. Let $\langle \a, \nu_{d,n}(\boldsymbol{x})\rangle\in \Z[\boldsymbol{x}]$ be a homogeneous polynomial in $n$ variables of degree $d$ with integer coefficients $\boldsymbol{a}$, where $\langle\cdot,\cdot\rangle$ denotes the inner product. For a non-singular form $P\in \mathbb{Z}[\boldsymbol{x}]$ of degree $k\ (\leq d)$ in $N$ variables, consider a set of integer vectors $\boldsymbol{a}\in \mathbb{Z}^N$, defined by $$\mathfrak{A}(A;P)=\{\boldsymbol{a}\in \Z^N:\ P(\boldsymbol{a})=0,\ \|\boldsymbol{a}\|_{\infty}\leq A\}.$$ By handling a new lattice problem via the geometry of numbers, we confirm that whenever $n> 24d$ and $d\geq 17,$ the proportion of integer coefficients $\boldsymbol{a}\in \mathfrak{A}(A;P)$, whose associated equation $f_{\boldsymbol{a}}(\boldsymbol{x})=0$ satisfies the Hasse principle, converges to $1$ as $A\rightarrow\infty$. 
This improves on the recent work of the second author.
\end{abstract}


\section{Introduction and statement of the results}\label{sec:intro}

In this paper, we are concerned with the solubility of \emph{random} Diophantine equations. To avoid ambiguity, we begin by introducing some notation. Let $\nu_{d,n}: \R^n\rightarrow \R^N$ denote the \emph{Veronese} embedding with $N=N_{d,n}:=\binom{n+d-1}{d}$, defined by listing all the monomials of degree $d$ in $n$ variables using the lexicographical ordering. Let $f_{\a}(\x):=\langle \a, \nu_{d,n}(\x)\rangle\in \Z[\x]$ be a homogeneous polynomial in $n$ variables of degree $d$ with integer coefficients $\a$, where $\langle\cdot,\cdot\rangle$ denotes the inner product. Lastly, we define 
\begin{equation*}
    \mathfrak{A}(A):=\{\a\in \Z^N:\ \|\a\|_{\infty}\leq A\}.
\end{equation*}
Browning, Sawin, and Le Boudec in \cite{RefJ3} proved that the proportion of $\a\in\Z^N$ in $ \mathfrak{A}(A)$, whose associated equation $f_{\a}(\x)=0$ satisfies the Hasse principle, converges to $1$ as $A\rightarrow \infty,$ provided that $n\geq d+1$ and $d\geq 4.$ Furthermore, Poonen and Voloch \cite{RefJ14} proved that the proportion of $\a\in\Z^N$ in $ \mathfrak{A}(A)$, whose associated equation $f_{\a}(\x)=0$ is everywhere locally soluble, converges to a positive constant $c$ as $A\rightarrow \infty,$ provided that $n\geq 2$ and $d\geq 2.$ Hence, combining these two results, we conclude that a positive proportion of homogeneous polynomial equations in $n$ variables of degree $d$ are soluble over $\Q$, provided that $n\geq d+1$ and $d\geq 4. $ 

The recent works \cite{RefJ10111,RefJ10012} of Lee, Lee, and the second author showed that such a conclusion on global solubility remains true even when the coefficients $\a\in \Z^N$ are constrained by a polynomial condition under a modest condition on the number of variables. More precisely, let $P$ be a non-singular form in $n$ variables of degree $k\geq2$. Define 
\begin{equation*}
    \mathfrak{A}(A;P):=\{\a\in \Z^N:\ P(\a)=0,\ \|\a\|_{\infty}\leq A\}.
\end{equation*}
The second author \cite{RefJ10012} proved that the proportion of $\a\in \Z^N$ in $ \mathfrak{A}(A;P)$, whose associated equation $f_{\a}(\x)=0$ satisfies the Hasse principle, converges to $1$ as $A\rightarrow \infty,$ provided that $n\geq 32d+17$, $d\geq 14$ and $k\leq d.$  Furthermore,  Lee, Lee and the second author \cite{RefJ10111} showed that  the proportion of $\a\in\Z^N$ in $ \mathfrak{A}(A;P)$, whose associated equation $f_{\a}(\x)=0$ is everywhere locally soluble, converges to a constant $c_P$ as $A\rightarrow \infty,$ provided that  $n\geq 2$, $d\geq 2$ and $k\leq C(n,d)$ for some constant $C(n,d)$. This constant $c_P$ is positive if there exists $\a\in \mathfrak{A}(A;P) $ such that $f_{\a}(\x)=0$ has an integer solution $\x\in \Z^n$ with $\nabla f_{\a}(\x)\neq \0.$

The main purpose of this paper is to show that such a conclusion still holds even when $n > 24d$, $d \ge 17$ and $k\leq d$. The crucial ingredient for this improvement is to handle a new lattice counting problem (see Lemma $\ref{lemma for a new lattice counting problem}$ below) which naturally arises in the argument. In the proof of [\ref{ref10012}, Lemma 3.6] we encounter exponential sums which lead to this same lattice counting problem, however, we simply used a trivial bound previously. In this paper, we obtain a much sharper upper bound utilizing tools from the geometry of numbers. Our hope is that the procedure described in this paper will inspire further improvements in the future.

In order to describe our main theorems, we temporarily pause here and provide some definitions. Recall that $f_{\a}(\x)$ is a homogeneous polynomial in $n$ variables of degree $d.$ Furthermore, for $\a\in \Z^N$ and $X>0$, we define 
\begin{equation*}
    \mathcal{I}_{\a}(X):=\{\x\in [1,X]^n\cap \Z^n:\ f_{\a}(\x)=0\}.
\end{equation*}
We note here that our argument proceeds for fixed $X>0$, and thus for simplicity, we write
\begin{equation}\label{def2.2}
    w=\log X
\end{equation}
and \begin{equation}\label{def2.3}
    W=\dprod_{p\leq w}p^{\lfloor\log w/\log p\rfloor}.
\end{equation}
Observe here that an application of the prime number theorem reveals that $\log W \leq 2w,$ which implies 
\begin{equation}
    W\leq X^{2}.
\end{equation}

For $L>0$ and $\a\in \Z^N,$ we define
\begin{equation}\label{def6.1}
    \sigma(\a;L)=L^{-(n-1)}\#\{\g\in [1,L]^n:\ f_{\a}(\g)\equiv 0\ \text{mod}\ L\}.
\end{equation}
We notice that by the Chinese remainder theorem one has
\begin{equation}\label{6.26.26.2}
    \sigma(\a;L)=\dprod_{p^r\|L}\sigma(\a;p^r).
\end{equation}
Then, on recalling the definition $(\ref{def2.3})$ of $W$, we write
\begin{equation}
    \mathfrak{S}_{\a}^*=\sigma(\a;W)=\dprod_{p^r\|W}\sigma(\a;p^r).
\end{equation}

Recall the definition $(\ref{def2.2})$ of $w.$ Put $\zeta=w^{-4-1/(8d)},$ and we introduce an auxiliary function 
$$\mathfrak{w}_{\zeta}(\beta)=\zeta\cdot\left(\frac{\text{sin}(\pi \zeta\beta)}{\pi \zeta\beta}\right)^2.$$
Note here that we chose $\zeta$ differently from \cite{RefJ10012} in order to optimize the result. 
This function has the Fourier transform
$$\widehat{\mathfrak{w}}_{\zeta}(\xi)=\int_{-\infty}^{\infty}\mathfrak{w}_{\zeta}(\beta)e(-\beta \xi)d\beta=\text{max}\{0,1-|\xi|/\zeta\}.$$
For $\a\in \Z^N$ and $A,X>0,$ we define 
\begin{equation}\label{defnJ*}
    \mathfrak{J}_{\a}^*:= \mathfrak{J}_{\a}^*(A,X)=A^{-1}X^{n-d}\dint_{[0,1]^n}\zeta^{-1} \widehat{\mathfrak{w}}_{\zeta}(A^{-1}f_{\a}(\boldsymbol{\gamma}))d\boldsymbol{\gamma}.
\end{equation}

\begin{defn}\label{defn4.2}
Let $n$ and $d$ be natural numbers with $d\geq 2.$ Consider the monomials of degree $d$ in $n$ variables $x_1,\ldots, x_n$. In particular, the number of these monomials is $N=\binom{n+d-1}{d}.$ Then, define $v_d(\x)\in \R^n$ and $w_d(\x)\in \R^{N-n}$ to be vectors associated with those monomials such that $(v_d(\x))_i$ is $x_i^d$ with $i=1,\ldots,n$ and $(w_d(\x))_j$'s are remaining monomials in lexicographical order with $j=1,\ldots, N-n$, respectively.
\end{defn}
For example, we find that 
$$v_3(x_1,x_2)=(x_1^3,x_2^3)\ \text{and}\ w_3(x_1,x_2)=(x_1^2x_2,x_1x_2^2).$$
Recall the definition of $f_{\a}(\x)$, that is $f_{\a}(\x)=\langle\a,\nu_{d,n}(\x)\rangle$. Let us define a permutation  $[\ \cdot\ ]:\ \Z^N\rightarrow \Z^N$ in the following way:   for a given $(\b,\c)\in \Z^N$ with $\b\in \Z^n$ and $\c\in \Z^{N-n}$,  the permutation $[\ \cdot\ ]$ is mapping $(\b,\c)\in \Z^N$ to $\a\in \Z^N$ such that
\begin{equation}\label{[]definition}
    f_{\a}(\x)=\langle\b,v_d(\x)\rangle+\langle\c,w_d(\x)\rangle.
\end{equation}

\begin{te}\label{thm2.2}
Let $A$ and $X$ be positive numbers and let $n$ and $d$ be natural numbers with $d\geq 4.$ Let $n=8s+r$ with $s\in \N$ and $1\leq r\leq 8.$ Suppose that $s\geq 3d$ and $X^{2d}\leq A\leq X^{s-d}.$ Suppose that $P\in \Z[\x]$ is a non-singular form of degree $k$ in $N_{d,n}$ variables. Then, whenever $N_{d,n}\geq 200k(k-1)2^{k-1}$, there is a positive number $\delta<1$ such that
\begin{equation*}
    \dsum_{\substack{\|\a\|_{\infty}\leq A\\P(\a)=0}}\left|\mathcal{I}_{\a}(X)-\mathfrak{S}_{\a}^*\mathfrak{J}_{\a}^*\right|^2\ll A^{N-4}X^{2n-2d}(\log A)^{-\delta}.
\end{equation*}
\end{te}

\bigskip

Recall the definition of $N:=N_{d,n}$. In advance of the statement of the following theorem, define a set $\mathcal{A}^{\text{loc}}_{d,n}(A;P)$ of integer vectors $\a\in \Z^N$ in $\mathfrak{A}(A;P)$ having the property that the associated equation  $f_{\a}(\x)=0$ is everywhere locally soluble.  
\begin{te}\label{thm2.3}
Let $A$ and $X$ be positive numbers with $X^3\leq A.$ Suppose that $n$ and $d$ are natural numbers with $n>d+1$ and $d\geq 2$. Suppose that $P\in \Z[\x]$ is a non-singular form of degree $k$ in $N_{d,n}$ variables. Then, whenever $N_{d,n}\geq 1000n^28^k$, one has
\begin{equation*}
    \#\left\{\a\in \mathcal{A}^{\text{loc}}_{d,n}(A;P):\ \begin{aligned}
        \mathfrak{S}^*_{\a}\mathfrak{J}_{\a}^*\leq X^{n-d}A^{-1}(\log A)^{-\eta}
    \end{aligned}\right\}\ll A^{N-2}\cdot(\log A)^{-\eta/(40n)},
\end{equation*}
for any $\eta>0.$
\end{te}
\begin{proof}
The only difference between Theorem \ref{thm2.3} and \cite[Theorem 1.2]{RefJ10012} is the choice of $\zeta.$ One readily sees that the choice of $\zeta=w^{-4-1/(8d)}$ does not harm the argument in \cite[Proposition 5.12]{RefJ10012}, and thus  \cite[Proposition 5.12]{RefJ10012} still holds with $\zeta=w^{-4-1/(8d)}$. Therefore,  we see by \cite[Section 6]{RefJ10012} that \cite[Theorem 1.2]{RefJ10012} holds with $\zeta=w^{-4-1/(8d)}$.
\end{proof}
\begin{te}\label{thm1.3}
Let $A$ and $X$ be positive numbers.  Suppose that  $A,X,n$ and $d$  satisfy the hypotheses in Theorem $\ref{thm2.2}$ and $\ref{thm2.3}$. Suppose that $P\in \Z[\x]$ is a non-singular form of degree $k$ in $N_{d,n}$ variables.  Then,  the proportion of integer vectors $\a\in \mathcal{A}_{d,n}^{\text{loc}}(A;P)$ in $\mathfrak{A}(A;P)$, having the property that $$ \mathcal{I}_{\a}(X)<A^{-1}X^{n-d}(\log A)^{-1/5},$$ converges to $0$ as $A\rightarrow \infty.$
\end{te}

\begin{proof}
 See the proof of \cite[Theorem 1.3]{RefJ10012}.
\end{proof}

\begin{coro}\label{coro1.5}
The conclusions of Theorem $\ref{thm2.2}$, $\ref{thm2.3}$ and $\ref{thm1.3}$ hold for $d\geq 17,$ $k\leq d$ and $n> 24d$ in place of the hypotheses on $n,d$ and $k$.
\end{coro}
    \begin{proof}
    It suffices to show that the conditions $d\geq 17,$ $k\leq d$ and $n>24d$ imply the hypotheses on $n,d$ and $k$ in Theorem $\ref{thm2.2},$ $\ref{thm2.3}$ and $\ref{thm1.3}.$
   For $d\geq 17,$ a modicum of computation reveals that we have
    $$1000\cdot 8^{d}\leq \frac{25^{d-2}}{d^2}.$$
Then, we see that whenever $d\geq 17$ and $n> 24d$, we obtain
    $$1000\cdot 8^{d}\leq \frac{1}{d^2}\cdot \left(\frac{n+d-1}{d}\right)^{d-2}.$$
  Hence, it follows that whenever $k\leq d$ one has
    \begin{equation*}
        \begin{aligned}
            1000\cdot n^2\cdot 8^k\leq 1000\cdot 8^{d} \cdot(n+d-1)^2\leq \left(\frac{n+d-1}{d}\right)^d\leq \binom{n+d-1}{d}=N_{d,n}.
        \end{aligned}
    \end{equation*}
    Furthermore, it implies that $N_{d,n}\geq 200k(k-1)2^{k-1}.$ Plus, if one writes $n=8s+r$ with $1\leq r\leq 8$, one sees that $s\geq 3d$ since $n>24d.$
    \end{proof}

Therefore, Theorem \ref{thm1.3} implies that the proportion of integer vectors $\a\in \Z^N$ in $\mathfrak{A}(A;P),$ whose associated equation $f_{\a}(\x)=0$ satisfies the Hasse principle, converges to $1$.
Meanwhile, by \cite{RefJ10111}, one finds that the proportion of integer vectors $\a\in \Z^N$ in $\mathfrak{A}(A;P)$, whose associated equation $f_{\a}(\x)=0$ is everywhere locally soluble, converges to a constant $c_P$ as $A\rightarrow \infty.$ Moreover, for each place of $v$ of $\mathbb{Q}$, if there exists a non-zero $\b_v\in \Q_v^N$ such that $P(\b_v)=0$ and the equation $f_{\b_v}(\x)=0$ has a solution $\x\in \Q_v^n$ satisfying $\nabla f_{\b_v}(\x)\neq \boldsymbol{0}$ (see also \cite{RefJ9} for $k=2$), the constant $c_P$ is positive. Combining this with Theorem $\ref{thm1.3}$, we conclude that the proportion of integer vectors $\a\in \Z^N$ in $\mathfrak{A}(A;P),$ whose associated equation $f_{\a}(\x)=0$ has a solution $\x\in \Q^n$, converges to $c_P$.

\subsection{Notation}

In this paper, we use bold symbol to denote vectors, and we write vectors by row vectors. For a given vector $\boldsymbol{v}\in \R^N$, we write $i$-th coordinate of $\boldsymbol{v}$ by $(\boldsymbol{v})_i$ or $v_i.$ Given that $n_1,n_2,\ldots,n_s\in \N$ with $n_1+n_2+\cdots+n_s=n$, we use notation $\x=(\x_1,\x_2,\ldots,\x_s)\in \R^n$ with $\x_1\in \R^{n_1}$, $\x_2\in \R^{n_2},$ $\ldots \x_s\in \R^{n_s}$ where $\x_1$ is a vector formed by the first $n_1$ coordinates of $\x$ and $\x_2$ is a vector formed by the next $n_2$ coordinates of $\x$ and so on. We use notation $\x^{(1)},\x^{(2)},\ldots,\x^{(l)}$ for 
$$\left(\dsum_{\x}\right)^s=\dsum_{\x^{(1)},\ldots,\x^{(l)}}.$$ 
We write $0\leq \x\leq X$ to abbreviate the condition $0\leq x_1,\ldots,x_s\leq X.$ Also, we preserve the summation conditions until different conditions are specified.

\subsection*{Acknowledgement}
Daniel Flores Galiote acknowledges the support of Purdue University, which provided funding for this research through the Ross-Lynn Research Scholar Fund. Kiseok Yeon is supported by the KAP allocation from the University of California, Davis. The authors thank Wooley for having us as his students and for his mentorship throughout the years. The authors also thank all of his contributions to the field of number theory.  Furthermore, the authors thank anonymous referees for providing helpful comments, which have improved the exposition of the paper.

\bigskip

\section{Auxiliary lemmas}\label{sec2}
In this section, we record some results from the geometry of numbers, and some results in $[\ref{ref10012}].$
Let $\Lambda$ be a sublattice of $\Z^n$ of rank $r$, and let $\b_1,\ldots,\b_r$ be its basis. Denote by $d(\Lambda)$ the determinant of the lattice $\Lambda$. It follows by \cite[Lemma IV.6A and 6D]{RefJ1445} that
\begin{equation*}
    d(\Lambda)^2=\dsum_{I}(\text{det}B_I)^2,
\end{equation*}
where $I$ runs over $r$-element subsets of $\{1,2,\ldots, n\},$ and $B_I$ denotes the $r\times r$-minor with rows indexed in $I$ of the matrix $B=(\b_1,\ldots,\b_r)$ formed with columns $\b_j.$
Define the orthogonal lattice $\Lambda^{\perp}:=\{\x\in \Z^n:\ \langle\x,\b_i\rangle=0\ (1\leq i\leq r)\}.$
 Define $$G(\Lambda):= \gcd_{I}\ \det B_I.$$ Then, we find from \cite[Lemma 2.1]{RefJ2} that
\begin{equation}
    d(\Lambda^{\perp})=d(\Lambda)/G(\Lambda).
\end{equation}

\bigskip

\begin{lem}\label{lemma for geometry of numbers 2}
    Let $\Lambda$ be a sublattice of rank $r$ in $\Z^n.$ Let $A\geq d(\Lambda).$ Then, the box $|\a|\leq A$ contains $O(A^r/d(\Lambda))$ elements of $\Lambda.$
\end{lem}
\begin{proof}
    See \cite[Lemma 1 (v)]{RefJ1454}.
\end{proof}

\bigskip

\begin{lem}\label{lemma for congruence bound}
    The number of pairs $(u, v) \in \Z^2$ with $|u| \leq B,|v| \leq B$ and $u^k \equiv v^k \bmod d$ does not exceed $O\left(B^{1+\epsilon}+B^{2+\epsilon} d^{\epsilon-2 / k}\right)$.
\end{lem}
\begin{proof}
    See \cite[Lemma 2.4]{RefJ2}.
\end{proof}

\bigskip

\begin{lem}\label{Bruedern-Dietmann lemma1} Let $U_s(A,X)$ be the number of integer solutions $0<|a_i|<A$ and $0\leq |x_i|,|z_i|\leq X$ with $A\geq X^d$ satisfying 
\begin{equation*}
    \dsum_{i=1}^{s}a_i(x_i^d-z_i^d)=0.
\end{equation*}
Then, we have
$$U_s(A,X)\ll A^sX^s+A^{s-1}X^{2s-d+\epsilon}.$$
\end{lem}
\begin{proof}
    See \cite[Theorem 2.5]{RefJ33}.
\end{proof}
\bigskip

To describe the next lemma, it is convenient to introduce some definitions. We define
\begin{equation*}
F_1(\alpha_1,\alpha_2,\h)=\dsum_{\substack{1\leq \x,\y\leq X\\ \x,\y\in \Z^n}}e(\alpha_1\langle \h,\nu_{d,n}(\x)-\alpha_2\langle\h, \nu_{d,n}(\y)\rangle)
\end{equation*}
and
\begin{equation*}
    F_2(\beta,\h)=\dsum_{\a\in I_{\h}}e(\beta(P(\a+\h)-P(\a))),
\end{equation*}
where
\begin{equation*}
    I_{\h}=\{\a\in \Z^N:\ \|\a\|_{\infty}\leq A,\ \|\a+\h\|_{\infty}\leq A\}.
\end{equation*}
Additionally, we define
\begin{equation*}
    G_1(\alpha_1,\alpha_2)=\dsum_{\substack{\|\h\|_{\infty}\leq 2A\\\h\in \Z^N}}|F_1(\alpha_1,\alpha_2,\h)|^2\ \text{and}\ G_2(\beta)=\dsum_{\substack{\|\h\|_{\infty}\leq 2A\\ \h\in \Z^N}}|F_2(\beta,\h)|^2.
\end{equation*}
For any measurable set $\mathfrak{B}\in [0,1),\a\in \Z^N$ and $X>1$, define
\begin{equation}\label{I_a(X,B)}
    \mathcal{I}_{\a}(X,\mathfrak{B})=\int_{\mathfrak{B}}\dsum_{1\leq x\leq X}e(\alpha f_{\a}(\x))d\alpha.
\end{equation}
\begin{lem}\label{lemma 2.3}
 For any measurable set $\mathfrak{B}\in [0,1)$, we have
    \begin{equation*}
\dsum_{\substack{\|\a\|_{\infty}\leq A\\ P(\a)=0}}|\mathcal{I}_{\a}(X,\mathfrak{B})|^2\ll X^n\int_{\mathfrak{B}^2}G_1(\alpha_1,\alpha_2)^{1/4}d\alpha_1 d\alpha_2\int_0^1G_2(\beta)^{1/4}d\beta.
    \end{equation*}
\end{lem}
\begin{proof}
    See \cite[Lemma 3.5]{RefJ10012}.
\end{proof}

\section{A lattice counting problem}\label{sec44}
In this section, we provide a new lattice counting problem which is the essential ingredient for the main theorem in this paper. 
Let $N(A,X)$ denote the number of solutions of equations 
\begin{equation}\label{key system of equations}
    \dsum_{1\leq i\leq s}a_i(x_i^d-z_i^d)=\dsum_{1\leq i\leq s}a_i(y_i^d-w_i^d)=0
\end{equation}
in integers $a_i,x_i,y_i,z_i,w_i$ with $|a_i|\leq A$ and $|x_i|,|y_i|,|z_i|,|w_i|\leq X$.

\begin{lem}\label{lemma for a new lattice counting problem}
Let $d \ge 4$ and $s\geq d+2.$ Suppose that $A\geq X^{2d}\geq 1.$ Then, one has
\begin{equation}\label{conclusion of lemma 3.1.}
    N(A,X)\ll A^s X^{2s}+A^{s-1}X^{3s-d+\eps}+A^{s-2}X^{4s-2d+\eps}E(X),
\end{equation}
where
\[E(X) = 1+X^{2d-s+3} +X^{4d-2s+4}.\]
\end{lem}
\begin{proof}
We begin by observing that it suffices to count the number of solutions of equations:
\begin{equation}\label{key system of equations2}
    \dsum_{1\leq i\leq s}a_i(x_i^d-z_i^d)=\dsum_{1\leq i\leq s}a_i(y_i^d-w_i^d)=0
\end{equation}
in integers $a_i,x_i,y_i,z_i,w_i$ with $0< |a_i|\leq A$ and $X/2 \leq |x_i|,|y_i|,|z_i|,|w_i|\leq X$. Indeed,  we temporarily define 
$$W(\alpha_1,\alpha_2):=\dsum_{|l|\leq A}\biggl|\dsum_{ |x|\leq X}e(\alpha_1 lx^d)\biggr|^2\biggl|\dsum_{|y|\leq X}e(\alpha_2ly^d)\biggr|^2.$$ Then, it follows by orthogonality that
\begin{equation}\label{N(A,X) representation}
    N(A,X)=\int_0^1\int_0^1 |W(\alpha_1,\alpha_2)|^{s}d\alpha_1d\alpha_2.
\end{equation}
By the elementary inequality $(a+b)^n\ll a^n+b^n$ with $a,b\geq 0$ and $n\in \mathbb{N},$ we deduce from $(\ref{N(A,X) representation})$ that
\begin{equation}\label{first upper bound for N(A,X)}
     N(A,X)\ll X^{4s}+A^sX^{2s}+\int_0^1\int_0^1| \widetilde{W}(\alpha_1,\alpha_2)|^sd\alpha_1 d\alpha_2,
\end{equation}
where 
\begin{equation*}
    \widetilde{W}(\alpha_1,\alpha_2):=\dsum_{0<|l|\leq A}\biggl|\dsum_{1\leq |x|\leq X}e(\alpha_1 lx^d)\biggr|^2\biggl|\dsum_{1\leq |y|\leq X}e(\alpha_2ly^d)\biggr|^2.
\end{equation*}
Let $\widetilde{W}_{ij}(\alpha_1,\alpha_2)$ be the portion of the sum $\widetilde{W}(\alpha_1,\alpha_2)$ where $2^{-i}X<|x|\leq 2^{-i+1}X$ and $2^{-j}X<|y|\leq 2^{-j+1}X$. Then, by applying the Cauchy-Schwarz inequality, one has
\begin{equation*}
    \widetilde{W}(\alpha_1,\alpha_2)\ll X^{\epsilon}\dsum_{i=1}^L\dsum_{j=1}^L \widetilde{W}_{ij}(\alpha_1,\alpha_2)
\end{equation*}
with $L=O(\log X).$ Then, it following by applying the H\"older's inequality in $(\ref{first upper bound for N(A,X)})$ that
\begin{equation}\label{second bound for N(A,X)}
     N(A,X)\ll X^{4s}+A^sX^{2s}+X^{\epsilon}\dsum_{i=1}^L\dsum_{j=1}^L \int_0^1\int_0^1| \widetilde{W}_{ij}(\alpha_1,\alpha_2)|^sd\alpha_1 d\alpha_2.
\end{equation}

The mean values over $\alpha_1$ and $\alpha_2$ in the third term of the right hand side in $(\ref{second bound for N(A,X)})$ count the number of solutions of a system of equations:
\begin{equation}\label{key system of equations3}
    \dsum_{1\leq k\leq s}a_k(x_k^d-z_k^d)=\dsum_{1\leq k\leq s}a_k(y_k^d-w_k^d)=0
\end{equation}
in integers $a_k,x_k,y_k,z_k,w_k$ with $0< |a_k|\leq A$ and $2^{-i}X < |x_k|,|z_k|\leq 2^{-i+1}X$ and $2^{-j}X <|y_k|,|w_k|\leq 2^{-j+1}X$. Meanwhile, the system $(\ref{key system of equations3})$ is equivalent to 
\begin{equation}\label{key system of equations4}
     \dsum_{1\leq k\leq s} a_k((2^{i-1}\cdot x_k)^d-(2^{i-1}\cdot z_k)^d)=\dsum_{1\leq k\leq s} a_k((2^{j-1}\cdot y_k)^d-(2^{j-1}\cdot w_k)^d)=0.
\end{equation}
Furthermore, one sees obviously that the number of solutions satisfying $(\ref{key system of equations4})$, with $0< |a_k|\leq A$ and $2^{-i}X < |x_k|,|z_k|\leq 2^{-i+1}X$ and $2^{-j}X <|y_k|,|w_k|\leq 2^{-j+1}X$, is bounded above by 
\begin{equation}\label{key system of equations5}
     \dsum_{1\leq k\leq s}a_k(x_k^d-z_k^d)=\dsum_{1\leq k\leq s} a_k(y_k^d-w_k^d)=0,
\end{equation}
in integers $a_i,x_i,y_i,z_i,w_i$ with $0< |a_i|\leq A$ and $X/2 < |x_k|,|y_k|,|z_k|,|w_k|\leq X$. This shows that 
\begin{equation*}
    \int_0^1\int_0^1| \widetilde{W}_{ij}(\alpha_1,\alpha_2)|^sd\alpha_1 d\alpha_2\leq  \int_0^1\int_0^1| \widetilde{W}_{11}(\alpha_1,\alpha_2)|^sd\alpha_1 d\alpha_2,
\end{equation*}
for all $1\leq i,j\leq L,$ and hence
it follows from $(\ref{second bound for N(A,X)})$ that
\begin{equation}\label{third bound for N(A,X)}
    N(A,X)\ll X^{4s}+A^sX^{2s}+X^{\epsilon}\int_0^1\int_0^1| \widetilde{W}_{11}(\alpha_1,\alpha_2)|^sd\alpha_1 d\alpha_2.
\end{equation}
On noting that the third term in $(\ref{third bound for N(A,X)})$ is bounded by the number of solutions satisfying $(\ref{key system of equations2})$, it suffices to show that the number of solutions satisfying $(\ref{key system of equations2})$ is bounded above by the right hand side of $(\ref{conclusion of lemma 3.1.}),$ in order to complete the proof of Lemma $\ref{lemma for a new lattice counting problem}$.

Let $N_1(A,X)$ be the number of solutions counted by $(\ref{key system of equations2})$  where
vectors $(x_1^d-z_1^d,\ldots, x_s^d-z_s^d)\in \R^s$ and $(y_1^d-w_1^d,\ldots,y_s^d-w_s^d)\in \R^s$ are linearly dependent over $\R.$ First we bound the number of solutions counted by $N_1(A,X)$ satisfying $x_i^d-z_i^d \neq 0$ for all $1 \le i \le s$.

The linearly dependent condition implies that there exists $c\in \Q$ such that 
\begin{equation}\label{conditino 1 L.D.}
    c(x_i^d-z_i^d)=(y_i^d-w_i^d)\ \text{with}\ 1\leq i\leq s.
\end{equation}

Consider the case $c=0.$ This gives $y_i^d=w_i^d$ for all $1\leq i\leq s.$ Hence, the system $(\ref{key system of equations2})$ of equations, satisfying $(\ref{conditino 1 L.D.})$ with $c=0$, implies that 
\begin{equation}\label{condition 2 L.D.}
    \sum_{i=1}^{s}a_i(x_i^d-z_i^d)=0\ \text{and}\ y_i^d=w_i^d\ (1\leq i\leq s).
\end{equation}
Therefore, since the number of solutions $y_i$ and $w_i$ satisfying $y_i^d=w_i^d$ with $|y_i|,|w_i|\leq X$ is $O(X^s),$ it follows by Lemma $\ref{Bruedern-Dietmann lemma1}$ that the number of solutions $x_i,y_i,z_i,w_i$ satisfying $(\ref{condition 2 L.D.})$ with $0<|a_i|\leq A$ and $X/2 \le |x_i|,|y_i|,|z_i|,|w_i|\leq X$ is $O(A^sX^{2s}+A^{s-1}X^{3s-d+\epsilon}).$ 

Consider the case $c\neq 0.$ Then, the system $(\ref{conditino 1 L.D.})$ of equations implies
\begin{equation}\label{condition 3 L.D.}
    (y_1^d-w_1^d)(x_i^d-z_i^d)=(y_i^d-w_i^d)(x_1^d-z_1^d)\ (2\leq i\leq s).
\end{equation}
Furthermore, the system $(\ref{key system of equations2})$ of equations with $x_i^d-z_i^d\neq 0$ for all $1\leq i\leq s$ implies that
\begin{equation}\label{condition 4 L.D.}
    \dsum_{i=1}^s a_i(x_i^d-z_i^d)=0\ \text{and}\ x_i^d-z_i^d\neq 0\ (1\leq i\leq s).
\end{equation}
On noting that $y_i^d-w_i^d=(y_i-w_i)(y_i^{d-1}+\cdots+w_i^{d-1}),$ we infer that for given $y_1,w_1$ and $a_i,x_i,z_i\ (1\leq i\leq s)$ satisfying  $(\ref{condition 4 L.D.})$, the number of solutions $y_i, w_i\ (2\leq i\leq s)$ is $O(X^{\epsilon})$ by the standard divisor estimate. Indeed, we ruled out the case $y_1=w_1,$ since otherwise $x_1^d=z_1^d$ by $(\ref{condition 3 L.D.})$ and the first equation in $(\ref{conditino 1 L.D.})$, which contradicts $x_i^d-z_i^d\neq 0\ (1\leq i\leq s).$ Therefore, on noting by Lemma $\ref{Bruedern-Dietmann lemma1}$ that the number of solutions $a_i,x_i,z_i\ (1\leq i\leq s)$ satisfying $(\ref{condition 4 L.D.})$ is $O(A^sX^{s}+A^{s-1}X^{2s-d+\epsilon}),$ and that the number of possible choices of $y_1,w_1$ is $O(X^2),$ we find by discussion above that the number of solutions $a_i,x_i,y_i,z_i,w_i\ (1\leq i\leq s)$ satisfying $(\ref{condition 3 L.D.})$ and $(\ref{condition 4 L.D.})$ is $O(A^s X^{s+2}+A^{s-1}X^{2s-d+2+\epsilon}).$ To sum up,  we conclude that the number of solutions counted by $N_1(A,X)$ with $x_1^d-z_i^d\neq 0$ for all $1\leq i\leq s$ is
\begin{equation}\label{conclusion N_1}
     O(A^sX^{2s}+A^{s-1}X^{3s-d+\epsilon}).
\end{equation}

Next, assume that the number of indices $i$ such that $x_i^d-z_i^d=0$ with $1\leq i\leq s$ is $r\ (\neq 0).$ We denote the set of these indices $i$ by $I$ with $|I|=r$. Then, the systems $(\ref{key system of equations2})$ and $(\ref{conditino 1 L.D.})$ of equations reduce to 
\begin{equation}\label{condition 5}
    \dsum_{i\notin I}a_i(x_i^d-z_i^d)=\dsum_{i\notin I}a_i(y_i^d-w_i^d)=0
\end{equation}
and 
\begin{equation}\label{condition 6}
    c(x_i^d-z_i^d)=y_i^d-z_i^d\ (i\notin I)
\end{equation}
and 
\begin{equation}\label{condition 7}
    y_i^d-w_i^d=0\ (i\in I).
\end{equation}
From the discussion leading to $(\ref{conclusion N_1})$, we find that the number of solutions $a_i,x_i,y_i,z_i,w_i\ (i\in I)$ satisfying $(\ref{condition 5})$ and $(\ref{condition 6})$ \begin{equation}\label{condition 8}
   O(A^{s-r}X^{2(s-r)}+A^{s-r-1}X^{3(s-r)-d+\epsilon}). 
\end{equation} Since the number of possible choices of $a_i\ (i\in I)$ is $O(A^r)$, and the number of solutions $x_i,y_i,w_i,z_i\ (i\in I)$ satisfying $(\ref{condition 7})$ and $x_i^d-z_i^d=0\ (i\in I)$ is $O(X^{2r})$,  we conclude by $(\ref{condition 8})$ that the number of solutions $a_i,x_i,y_i,z_i,w_i$ satisfying $(\ref{condition 5}), (\ref{condition 6})$ and $(\ref{condition 7})$ together with $x_i^d-z_i^d=0\ (i\in I)$ is $O(A^sX^{2s+\epsilon}+A^{s-1}X^{3s-r-d+\epsilon}).$ Therefore, by summing over the cases with $0\leq r\leq s$, we conclude that is 
\begin{equation}\label{eqs:N1bound}
    N_1(A,X)\ll A^sX^{2s}+A^{s-1}X^{3s-d+\epsilon}.
\end{equation}

We denote by $N_2(A,X)$ the  number of solutions counted by  ($\ref{key system of equations2}$) where the vectors $(x_1^d-z_1^d,\ldots, x_s^d-z_s^d)\in \R^s$ and $(y_1^d-w_1^d,\ldots,y_s^d-w_s^d)\in \R^s$ are linearly independent over $\R.$ We bound this quantity via an application of the geometry of numbers. For fixed $\bfx,\bfy,\bfz,\bfw$ contributing to $N_2(A,X)$ we denote
\[\Delta_{i,j} = \det \left(\begin{matrix}
    x_i^d-z_i^d & y_i^d-w_i^d \\
    x_j^d-z_j^d & y_j^d-w_j^d
\end{matrix}\right).\]
Then, just as in \cite[Lemma 2.5]{RefJ2} we have by Lemma \ref{lemma for geometry of numbers 2} the following bound.
\begin{align}\label{eqs:reductionindependent}
    N_2(A,X) &\ll A^{s-2}\sum_{\substack{X/2 \le \bfx,\bfy,\bfz,\bfw \le X \\ \Delta_{i,j} \neq 0 \\ \text{for some }1\le i<j \le s}}\frac{\gcd_{1 \le i<j \le s}\Delta_{i,j}}{\max_{1 \le i < j \le s}\Delta_{i,j}} \\
    &\ll A^{s-2} \sum_{\substack{X/2 \le x_1,y_1,z_1,w_1 \le X \\ X/2 \le x_2,y_2,z_2,w_2 \le X \\ \Delta_{1,2} \ge 1}} \frac{1}{\Delta_{1,2}}\sum_{D | \Delta_{1,2}}D \psi_{s}(X,D),
\end{align}
where $\psi_{s}(X,D) = \psi_{s}(X,D;x_1,x_2,y_1,y_2,z_1,z_2,w_1,w_2)$ denotes the number of solutions 
\[x_3,\ldots,x_s,y_3,\ldots,y_s,z_3,\ldots,z_s,w_3,\ldots,w_s \in \N \cap [X/2,X],\] 
to the system of congruences
\[(x_i^d-z_i^d)(y_j^d-w_j^d) \equiv (x_j^d-z_j^d)(y_i^d-w_i^d) \nmod D \quad \text{for all} \quad 1 \le i < j \le s.\]

We prove a bound on $\psi_s(X,D)$ when $1 \le D \le 8X^{2d}$ by defining the quantity 
\[m(\bfx,\bfy,\bfz,\bfw) = \min_{3 \le i,j \le s} \{(D,x_i^d-z_i^d),(D, y_j^d-w_j^d)\}, \]
and consider those solutions counted by $\psi_s(X,D)$ satisfying $m(\bfx,\bfy,\bfz,\bfw) = l$, we denote the quantity of these solutions by $\psi_{s,l}(X,D)$. Then by dyadically summing over the range $1 \le l \le D$ one deduces that
\begin{equation}\label{eqs:psibound}
    \psi_s(X,D) \ll X^\eps \sup_{1 \le L \le D} \sum_{L/2 \le l \le L} \psi_{s,l}(X,D).
\end{equation}

Given a fixed $1 \le l \le D$ we may suppose, without loss of generality, that $(D,x_3^d-z_3^d) = l$. Similar to the method of Br\"{u}dern and Dietmann, we ignore most restrictions and simply bound the number of solutions satisfying the congruences
\[(x_3^d-z_3^d)(y_j^d-w_j^d) \equiv (x_j^d-z_j^d)(y_3^d-w_3^d) \nmod D \quad \text{for all} \quad j \neq 3,\]
with $m(\bfx,\bfy,\bfz,\bfw) = l$.

First, let us bound the number of choices for $x_3,\ldots,x_s,z_3,\ldots,z_s$. In order for a solution to be counted in $\psi_{s,l}(X,D)$ it must satisfy
\[x_j^d-z_j^d \equiv 0 \nmod T_j \ \text{for all}\ 3 \le j \le s,\]
where $T_j\ge l$. Recalling Lemma \ref{lemma for congruence bound}, we deduce that the number of choices for $x_3,\ldots,x_s,z_3,\ldots,z_s$ is at most
\begin{equation}\label{eqs:xzbound}
    X^{s-2+\eps} \left(1+X l^{-\frac{2}{d}}\right)^{s-2}.
\end{equation}
We apply the same procedure to bound the number of choices for $y_3,w_3$, thus the contribution from these two variables is at most
\begin{equation}\label{eqs:y3w3bound}
    X^{1+\eps}\left(1+X l^{-\frac{2}{d}} \right).
\end{equation}

Next, with the variables $\bfx,y_3,\bfz,w_3$ now fixed, we bound the number of choices for $y_4,\ldots,y_s,w_4,\ldots,w_s$ in two ways. Note that $y_4,\ldots,y_s,w_4,\ldots,w_s$ must satisfy the following family of congruences
\begin{equation}\label{eqs:ywcongprior}
    (x_3^d-z_3^d)(y_j^d-w_j^d) \equiv (x_j^d-z_j^d)(y_3^d-w_3^d) \nmod D \ \text{for all}\ 4 \le j \le s.
\end{equation}
By our definition of $l$ one may deduce that there exists fixed integers $C_j$ for each $4 \le j \le s$ such
that the system of congruences (\ref{eqs:ywcongprior}) is equivalent to the following system of congruences
\begin{equation}\label{eqs:ywcong}
    y_j^d-w_j^d \equiv C_j \nmod D/l \ \text{for all}\ 4 \le j \le s.
\end{equation}
By considering the exponential sums which count solutions of the system (\ref{eqs:ywcong}) it is easily seen that the number of solutions to (\ref{eqs:ywcong}) is bounded above by the number of solutions to
\[y_j^d-w_j^d \equiv 0 \nmod D/l \ \text{for all}\ 4 \le j \le s.\]
Applying Lemma \ref{lemma for congruence bound} $s-3$ times, we deduce that the number of choices for $y_4,\ldots,y_s,w_4,\ldots,w_s$ is at most 
\begin{equation}\label{eqs:ywfirstbound}
    X^{s-3+\eps} \left(1+X (D/l)^{-\frac{2}{d}} \right)^{s-3}.
\end{equation}
However, one may also apply the same method we used for the variables $\bfx,y_3,\bfz,w_3$ and obtain the bound
\begin{equation}\label{eqs:ywsecondbound}
    X^{s-3+\eps} \left(1+X l^{-\frac{2}{d}} \right)^{s-3}.
\end{equation}

Combining (\ref{eqs:xzbound}), (\ref{eqs:y3w3bound}), (\ref{eqs:ywfirstbound}), and (\ref{eqs:ywsecondbound}) we deduce that when $L/2 \le l \le L$ one has that
\begin{equation}\label{eqs:psislbound}
    \psi_{s,l}(X;D) \ll X^{2s-4+\eps}\left(1+XL^{-\frac{2}{d}} \right)^{s-1}\left(1+X \min\{L^{-\frac{2}{d}},(D/L)^{-\frac{2}{d}}\} \right)^{s-3}.
\end{equation}
We now list different bounds on $\psi_{s,l}(X;D)$ by considering different cases relating the relative sizes of $L$ to $X^{d/2},D^{1/2},$ and $DX^{-d/2}$. 
\begin{align*}
    &X^{2s-4+\eps}(XL^{-\frac{2}{d}})^{s-1}, &\text{if}& \ 1 \le L \le \min\{X^{\frac{d}{2}},D^{\frac{1}{2}},DX^{-\frac{d}{2}}\}, \\
    &X^{2s-4+\eps}(XL^{-\frac{2}{d}})^{s-1} (X(D/L)^{-\frac{2}{d}})^{s-3}, &\text{if}& \ DX^{-\frac{d}{2}} \le L \le \min\{X^{\frac{d}{2}},D^{\frac{1}{2}}\}, \\
    &X^{2s-4+\eps}(XL^{-\frac{2}{d}})^{2s-4}, &\text{if}& \ D^{\frac{1}{2}} \le L \le X^{\frac{d}{2}}, \\ 
    &X^{2s-4+\eps}, &\text{if}& \ X^{\frac{d}{2}} \le L \le D. \\
\end{align*}
Combining this with (\ref{eqs:psibound}) and noting that $2 \le \frac{2}{d}(s-2)$ we deduce
\begin{equation}\label{eqs:finalpsibound}
    \psi_s(X;D) \ll X^{4s-8+\eps}D^{\frac{1}{2}-\frac{2}{d}(s-2)}+X^{3s-5+\eps}+X^{2s-4+\eps}D.
\end{equation}
Utilizing (\ref{eqs:finalpsibound}) with (\ref{eqs:reductionindependent}), we deduce
\[N_2(A,X) \ll A^{s-2} X^{\eps}\left(X^{2s+2d+4} +X^{3s+3}+ X^{4s-8}\Xi_d(X)\right),\]
where
\[\Xi_d(X) = \sum_{\substack{X/2 \le x_1,y_1,z_1,w_1 \le X \\ X/2 \le x_2,y_2,z_2,w_2 \le X \\ \Delta_{1,2} \ge 1}} \frac{1}{\Delta_{1,2}}.\]

We now establish that 
\begin{equation}\label{claim}
    \Xi_d(X)\ll X^{8-2d+\epsilon}.
\end{equation}
Since $\Delta_{1,2}\geq 1$, one of terms $(x_1^d-z_1^d)(y_2^d-w_2^d)$ and $(x_2^d-z_2^d)(y_1^d-w_1^d)$ is non-zero. Then, it suffices to consider the case 
 $(x_2^d-z_2^d)(y_1^d-w_1^d)\neq 0$ in the summation in $(\ref{claim}).$
Observe that 
\begin{equation}
\begin{aligned}
    \min_{\substack{X/2\leq |x_2|,|\widetilde{x}_2|\leq X\\ x_2\neq \widetilde{x}_2}}|(x_2^d-z_2^d)-(\widetilde{x}_2^d-z_2^d)|=\min_{\substack{X/2\leq |x_2|,|\widetilde{x}_2|\leq X\\ x_2\neq \widetilde{x}_2}}|x_2^d-\widetilde{x}_2^d|\geq B,
\end{aligned}
\end{equation}
with $B=O(X^{d-1}).$
Then, by observing that for fixed $z_2$, the function $x^d-z_2^d$ is monotonic increasing function in $x$, we infer that 
\begin{equation*}
\begin{aligned}
      \sum_{\substack{x_1,x_2 \\ y_1,y_2 \\ z_1,z_2 \\ w_1,w_2\\ \Delta_{1,2}\geq 1}}\frac{1}{\Delta_{1,2}}&= \dsum_{\substack{x_1,z_1\\ y_2,w_2}}\dsum_{z_2,y_1,w_1}\dsum_{\substack{x_2\\ \Delta_{1,2}\geq 1}}\frac{1}{|(x_1^d-z_1^d)(y_2^d-w_2^d)-(x_2^d-z_2^d)(y_1^d-w_1^d)|}\\
      &\ll \dsum_{\substack{x_1,z_1\\ y_2,w_2}}\dsum_{z_2,y_1,w_1}\dsum_{\substack{1\leq n\leq X}} \frac{1}{1+nB(y_1^d-w_1^d)}\\
      &\ll X^5\dsum_{y_1,w_1}\dsum_{1\leq n\leq X}\frac{1}{1+nB(y_1^d-w_1^d)}.
\end{aligned}
\end{equation*}
Similarly, the last expression is seen to be 
\begin{equation*}
\begin{aligned}
    &X^5\dsum_{y_1,w_1}\dsum_{1\leq n\leq X}\frac{1}{1+nB(y_1^d-w_1^d)}\\
    &\ll X^5 \dsum_{1\leq n\leq X}\dsum_{w_1}\dsum_{1\leq m\leq X}\frac{1}{1+nmB^2}\\
    &\ll X^6\dsum_{1\leq n,m\leq X}\frac{1}{1+nmB^2}.
\end{aligned}
\end{equation*}
It follows by the standard divisor estimate that the bound in the last expression is 
\begin{equation*}
    \ll X^{6+\epsilon}\dsum_{1\leq n\leq X^2}\frac{1}{1+nB^2}\ll X^{6+\epsilon}B^{-2}.
\end{equation*}
Since $B=O(X^{d-1}),$ we confirm the claim $(\ref{claim}).$ We conclude that
\begin{align}\label{eqs:N2bound}
    N_2(A,X) \ll A^{s-2}X^{4s-2d+\eps}\left(1+X^{2d-s+3} +X^{4d-2s+4}\right).
\end{align}

Combining (\ref{eqs:N1bound}) with (\ref{eqs:N2bound}) we deduce the claimed bound.

\end{proof}

\section{Bounds for large moduli}\label{sec4}
In this section, we provide bounds for large moduli via mean values of exponential sums that coincide with the counting problem in Lemma $\ref{lemma for a new lattice counting problem}$. We define here the major and minor arcs. 
For $B>0,$ define the major arcs 
\begin{equation}\label{4.24.24.2}
    \mathfrak{M}(B)=\bigcup_{\substack{0\leq a\leq q\leq B\\(q,a)=1}}\mathfrak{M}(q,a),
\end{equation}
where $$\mathfrak{M}(q,a)=\{\alpha\in [0,1):\ |\alpha-a/q|\leq BA^{-1}X^{-d}\},$$ 
and define the minor arcs $\mathfrak{m}(B)=[0,1)\setminus \mathfrak{M}(B).$ Recall from (\ref{def2.2}) that $w = \log X$. Here and throughout, we abbreviate $\mathfrak{M}(w)$ and $\mathfrak{m}(w)$ simply to $\mathfrak{M}$ and $\mathfrak{m}.$

\bigskip

In preparation for the next Lemma, we recall the definition of the permutation $[\ \cdot\ ]$ from Section \ref{sec:intro}, the definition of 
\[\mathcal{I}_{\a}(X,\mathfrak{B})=\mathcal{I}_{[(\b,\c)]}(X,\mathfrak{B}),\]
from $(\ref{I_a(X,B)})$, and the quantity $N=\binom{n+d-1}{d}.$

\begin{lem}\label{lem4.6}
Let $n=8s+r$ with $s,r\in \N$ and  $s\geq 3d$ and $1\leq r\leq 8$ Suppose that $P\in \Z[\x]$ is a non-singular form in $N$ variables of degree $k\geq 2.$ Then, whenever $1\leq X^{2d}\leq A\leq X^{s-d}$ and $N\geq 200k(k-1)2^{k-1}$, there exists $\delta'>0$ such that we have 
\begin{equation*}
\dsum_{\substack{\|\b\|_{\infty},\|\c\|_{\infty}\leq A\\ P([\b,\c])=0}}\bigl|\mathcal{I}_{[(\b,\c)]}(X,\mathfrak{m}(X^{\delta}))\bigr|^2\ll A^{N-k-2-\delta'}X^{2n-2d}.
\end{equation*}
\end{lem}
\begin{proof}
It follows by Lemma $\ref{lemma 2.3}$ that 
\begin{equation}\label{beginproof}
\begin{aligned}
&\dsum_{\substack{\|\b\|_{\infty},\|\c\|_{\infty}\leq A\\ P([\b,\c])=0}}\bigl|\mathcal{I}_{[(\b,\c)]}(X,\mathfrak{m}(X^{\delta}))\bigr|^2\\
&\ll X^n\int_{\mathfrak{m}(X^{\delta})^2}G_1(\alpha_1,\alpha_2)^{1/4}d\alpha_1 d\alpha_2\int_0^1G_2(\beta)^{1/4}d\beta.
\end{aligned}
\end{equation}
Recall the definition of $G_1(\alpha_1,\alpha_2)$, that is
\begin{equation*}
G_1(\alpha_1,\alpha_2)=\dsum_{\substack{\|\h\|_{\infty}\leq 2A\\\h\in \Z^N}}|F_1(\alpha_1,\alpha_2,\h)|^2,
\end{equation*}
where 
\begin{equation*}
F_1(\alpha_1,\alpha_2,\h)=\dsum_{\substack{1\leq \x,\y\leq X\\ \x,\y\in \Z^n}}e(\alpha_1\langle \h,\nu_{d,n}(\x)-\alpha_2\langle\h, \nu_{d,n}(\y)\rangle).
\end{equation*}
On recalling the definition of the permutation $[\ \cdot\ ]$, we write $\h=[(\boldsymbol{l},\m)]$ with $\boldsymbol{l}\in \Z^n$ and $\m\in \Z^{N-n}.$ Then, by squaring out and changing the order of summations, one deduces by the triangle inequality that
\begin{equation*}
   G_1(\alpha_1,\alpha_2)\ll A^{N-n}\dsum_{\substack{1\leq \x^{(1)},\x^{(2)}\leq X\\ 1\leq \y^{(1)}, \y^{(2)}\leq X\\ \x^{(i)},\y^{(i)}\in \Z^n}}\biggl|\dsum_{\substack{\|\boldsymbol{l}\|_{\infty}\leq 2A\\\boldsymbol{l}\in \Z^n}}e(\Psi(\alpha_1,\alpha_2,\boldsymbol{l},\x^{(1)},\x^{(2)},\y^{(1)},\y^{(2)}))\biggr|,
\end{equation*}
where
\begin{equation*}
\begin{aligned}
&\Psi(\alpha_1,\alpha_2,\boldsymbol{l},\x^{(1)},\x^{(2)},\y^{(1)},\y^{(2)})\\
&=\alpha_1\langle\boldsymbol{l},v_d(\x^{(1)})-v_d(\x^{(2)})\rangle-\alpha_2\langle\boldsymbol{l},v_d(\y^{(1)})-v_d(\y^{(2)})\rangle.
\end{aligned}
\end{equation*}
By applying the Cauchy-Schwarz inequality and the triangle inequality, one sees that
\begin{equation}\label{G(alpha1,alpha2)}
    G_1(\alpha_1,\alpha_2)\ll A^{N-n}(X^{4n})^{1/2}(A^nH(\alpha_1,\alpha_2))^{1/2},
\end{equation}
where 
\begin{equation*}
H(\alpha_1,\alpha_2)=\dsum_{\boldsymbol{l}}\biggl|\dsum_{\substack{\x^{(1)},\x^{(2)}\\ \y^{(1)},\y^{(2)}}}e(\Psi(\alpha_1,\alpha_2,\boldsymbol{l},\x^{(1)},\x^{(2)},\y^{(1)},\y^{(2)}))\biggr|.
\end{equation*}
Note that
\begin{equation}\label{Halpha1alpha2}
H(\alpha_1,\alpha_2)=\biggl(\dsum_{-2A\leq l\leq 2A}\biggl|\dsum_{1\leq x\leq X}e(\alpha_1 lx^d)\biggr|^2\biggl|\dsum_{1\leq y\leq X}e(\alpha_2ly^d)\biggr|^2\biggr)^n.
\end{equation}
For simplicity, we temporarily write $H(\alpha_1,\alpha_2)=W(\alpha_1,\alpha_2)^n,$
where $$W(\alpha_1,\alpha_2)=\dsum_{-2A\leq l\leq 2A}\biggl|\dsum_{1\leq x\leq X}e(\alpha_1 lx^d)\biggr|^2\biggl|\dsum_{1\leq y\leq X}e(\alpha_2ly^d)\biggr|^2.$$  
By substituting $(\ref{Halpha1alpha2})$ into $(\ref{G(alpha1,alpha2)})$ and that into $(\ref{beginproof})$, we deduce that
\begin{equation}\label{4.54.5}
\begin{aligned}
&\dsum_{\substack{\|\b\|_{\infty},\|\c\|_{\infty}\leq A\\ P([\b,\c])=0}}\bigl|\mathcal{I}_{[(\b,\c)]}(X,\mathfrak{m}(X^{\delta}))\bigr|^2\\
&\ll A^{N/4-n/8}X^{3n/2}\int_{\mathfrak{m}(X^{\delta})^2}H(\alpha_1,\alpha_2)^{1/8}d\alpha_1 d\alpha_2 \cdot \int_0^1 G_2(\beta)^{1/4}d\beta\\
&\ll A^{N-k-n/8}X^{3n/2}\int_{\mathfrak{m}(X^{\delta})^2}H(\alpha_1,\alpha_2)^{1/8}d\alpha_1 d\alpha_2,
\end{aligned}
\end{equation}
where we have used [$\ref{ref10012}$, Lemma 2.10] with $l=1, B=0, A_1=A_2=2A$ and $\sigma=1/4$ that whenever $N\geq 200k(k-1)2^{k-1},$ one has
\begin{equation*}
    \int_0^1 G_2(\beta)^{1/4}d\beta\ll A^{3N/4-k}.
\end{equation*}
Furthermore, it follows by applying H\"{o}lder's inequality that the bound in the last expression in $(\ref{4.54.5})$ is 
\begin{equation}\label{bound for 4.5}
    \ll A^{N-k-n/8}X^{3n/2}\cdot \int_{[0,1)^2}W(\alpha_1,\alpha_2)^{s}d\alpha_1 d\alpha_2\cdot \sup_{(\alpha_1,\alpha_2)\in \mathfrak{m}(X^{\delta})^2}|W(\alpha_1,\alpha_2)|^{r/8}.
\end{equation}
By the trivial estimate $|\sum_{1\leq y\leq X}e(\alpha_2ly^d)|^2\ll X^2,$ we deduce by [$\ref{ref2}$, Lemma 4.3] that
\begin{equation}\label{pointwise estimate}
  \sup_{(\alpha_1,\alpha_2)\in \mathfrak{m}(X^{\delta})^2}  W(\alpha_1,\alpha_2)\ll AX^{4-\eta},
\end{equation}
with $\eta>0.$
Furthermore, since the mean value $\int_{[0,1)^2}W(\alpha_1,\alpha_2)^{s}d\alpha_1 d\alpha_2$ is bounded above by the number of integer solutions with $|a_i|\leq 2A$ and $|x_i|,|y_i|,|z_i|,|w_i|\leq X$ satisfying $(\ref{key system of equations})$, we find by Lemma $\ref{lemma for a new lattice counting problem}$ that whenever $1\leq X^{2d}\leq A\leq X^{s-d}$ one has
\begin{equation}\label{mean value of exponential sums}
\int_{[0,1)^2}W(\alpha_1,\alpha_2)^{s}d\alpha_1 d\alpha_2\ll A^{s-2+\epsilon}X^{4s-2d}.
\end{equation}
On substituting $(\ref{pointwise estimate})$ and $(\ref{mean value of exponential sums})$ into $(\ref{bound for 4.5})$, we conclude by $(\ref{4.54.5})$ that
\begin{equation*}
\begin{aligned}
\dsum_{\substack{\|\b\|_{\infty},\|\c\|_{\infty}\leq A\\ P([\b,\c])=0}}\bigl|\mathcal{I}_{[(\b,\c)]}(X,\mathfrak{m}(X^{\delta}))\bigr|^2\ll A^{N-k-2}X^{2n-2d-\delta'},
\end{aligned}
\end{equation*}
for some $\delta'>0.$ Since $A\leq X^{n-d}$, we complete the proof of Lemma $\ref{lem4.6}.$

\end{proof}

\section{Bounds for small moduli}\label{sec5}
In this section, we provide bounds for small moduli using the same argument as in [$\ref{ref10012}$, Lemma 3.8]. However, for the success of the argument with the relaxed bound on the number of variables $n$, we slightly modify the argument.

\begin{lem}\label{lemma 5.1}
With the same condition in Lemma $\ref{lem4.6}$, whenever $1\leq X^{2d}\leq A\leq X^{s-d}$ and $N\geq 200k(k-1)2^{k-1}$, there exists $\delta'>0$ such that we have 
\begin{equation*}    \dsum_{\substack{\|\b\|_{\infty},\|\c\|_{\infty}\leq A\\ P([\b,\c])=0}}\bigl|\mathcal{I}_{[(\b,\c)]}\left(X,\mathfrak{M}(X^{\delta})\setminus \mathfrak{M}(\log X)\right)\bigr|^2\ll A^{N-k-2}X^{2n-2d}(\log A)^{-\eta},
\end{equation*}
for some $\eta>0.$
\end{lem}
\begin{proof}
On observing that 
\begin{equation*}
\begin{aligned}
    &\mathfrak{M}(X^{\delta})\setminus \mathfrak{M}(\log X)\\
    &=(\mathfrak{M}(X^{\delta})\setminus \mathfrak{M}((\log X)^{9d}))\cup (\mathfrak{M}((\log X)^{9d})\setminus \mathfrak{M}(\log X) ) \\
   & =\biggl(\bigcup_{j=0}^{J_1}(\mathfrak{M}(2^{j+1}(\log X)^{9d})\setminus\mathfrak{M}(2^j(\log X)^{9d}))\biggr) \\
   &\ \ \ \ \ \ \ \ \ \ \ \ \ \ \ \ \ \ \ \ \ \ \ \ \ \ \ \ \ \ \ \ \ \ \ \ \ \ \ \ \cup \biggl(\bigcup_{j=0}^{J_2}(\mathfrak{M}(2^{j+1}\log X)\setminus\mathfrak{M}(2^j\log X))\biggr)
\end{aligned}
\end{equation*}
with $J_1=O(\log X)$ and $J_2=O(\log \log X),$ we deduce by the Cauchy-Schwarz inequality that
\begin{equation}\label{4.16}
    \begin{aligned}
     & \dsum_{\substack{\|\a\|_{\infty}\leq A\\ P(\a)=0}}\left|\mathcal{I}_{\a}(X,\mathfrak{M}(X^{\delta})\setminus \mathfrak{M}(\log X))\right|^2\\
     &\ll J_1\dsum_{j=0}^{J_1} \dsum_{\substack{\|\a\|_{\infty}\leq A\\ P(\a)=0}}\left|\mathcal{I}_{\a}(X,\mathfrak{M}(2^{j+1}(\log X)^{9d})\setminus\mathfrak{M}(2^j(\log X)^{9d}))\right|^2\\
     &\ \ \ \ \ \ \ \ \ \ \ \ \ \ \ \ \ \ \ \ \ \ \ \ \ + J_2\dsum_{j=0}^{J_2} \dsum_{\substack{\|\a\|_{\infty}\leq A\\ P(\a)=0}}\left|\mathcal{I}_{\a}(X,\mathfrak{M}(2^{j+1}\log X)\setminus\mathfrak{M}(2^j\log X))\right|^2.
    \end{aligned}
\end{equation}

Now, we analyze the mean value
$$ \dsum_{\substack{\|\a\|_{\infty}\leq A\\ P(\a)=0}}\left|\mathcal{I}_{\a}(X,\mathfrak{M}(2Q)\setminus\mathfrak{M}(Q))\right|^2,$$
with $\log X\leq Q\leq X^{\delta}.$ For simplicity, we temporarily write
\begin{equation}
    \mathfrak{C}=\mathfrak{M}(2Q)\setminus\mathfrak{M}(Q).
\end{equation}
Then, by [$\ref{ref10012}$, Proposition 3.4], we deduce that
\begin{equation}\label{4.18}
     \dsum_{\substack{\|\a\|_{\infty}\leq A\\ P(\a)=0}}\left|\mathcal{I}_{\a}(X,\mathfrak{C})\right|^2\ll A^{N-n/8-k}X^{7n/4}\dint_{\mathfrak{C}}T(\alpha_1)^{\lfloor n/2\rfloor/8}d\alpha_1\dint_{\mathfrak{C}}T(\alpha_2)^{\lceil n/2\rceil/8}d\alpha_2,
\end{equation}
where $T(\alpha)=\sum_{-A\leq b\leq A}\bigl|\sum_{1\leq x\leq X}e(b\alpha x^d)\bigr|^2$ with $\alpha\in \R.$

Meanwhile, whenever $\alpha\in\mathfrak{C},$ there exist $q\in \N$ and $a\in \Z$ with $(q,a)=1$ such that $Q\leq q\leq 2Q$ and
$$|\alpha-a/q|\leq 2Q(AX^d)^{-1}.$$
Also, for $b\in \Z\setminus \{0\}$ with $|b|\leq A$, if we write $l=(q,b)$, $\widetilde{q}=q/l$ and $\widetilde{b}=b/l$, we have
\begin{equation*}
    \left|b\alpha-\frac{a\widetilde{b}}{\widetilde{q}}\right|\leq \frac{2Qb}{AX^d}\leq \frac{2Q}{X^d}.
\end{equation*}
Hence, since $\widetilde{q}\leq 2Q$, it follows by [$\ref{ref15}$, Lemma 2.7] that
\begin{equation}\label{4.19}
    \dsum_{1\leq x\leq X}e(b\alpha x^d)-(\widetilde{q})^{-1}S(\widetilde{q},a\widetilde{b})v(\beta)=O((2Q)^2),
\end{equation}
where $S(q,a)=\dsum_{n=1}^qe(an^d/q)$ and $v(\beta)=\dint_0^Xe(\beta \gamma^d)d\gamma$ with $\beta=b\alpha-\frac{a\widetilde{b}}{\widetilde{q}}.$ Thus, when $\alpha\in \mathfrak{C}$, we deduce from ($\ref{4.19}$) that 
\begin{equation}\label{4.20}
\begin{aligned}
    T(\alpha)&\ll X^2+\dsum_{l|q}\dsum_{\substack{|b|\leq A\\ (q,b)=l}}\biggl|\dsum_{1\leq x\leq X}e(b\alpha x^d)\biggr|^2\\
    &\ll X^2 +\dsum_{l|q}\dsum_{\substack{|b|\leq A\\ (q,b)=l}}
    \left(|((\widetilde{q})^{-1}S(\widetilde{q},a\widetilde{b})v(\beta)|^2+(2Q)^4\right).
\end{aligned}
\end{equation}

By [$\ref{ref15}$, Theorem 4.2], we have a bound $S(q/l,a(b/l))\ll (q/l)^{1-1/d}$
and a trivial bound $v(\beta)\leq X$. Hence, on substituting these estimates into $(\ref{4.20})$, we find that
\begin{equation}\label{4.21}
    T(\alpha)\ll X^2+\dsum_{l|q}AX^2q^{-2/d}l^{2/d-1}+AQ^4.
\end{equation}
The second term $\sum_{l|q}AX^2q^{-2/d}l^{2/d-1}$ is bounded above by $AX^2q^{-2/d}\sum_{l|q}1,$ and by the standard divisor estimate, this bound is $O(AX^2q^{-2/d+\epsilon})$. Recall that $\alpha\in \mathfrak{C}$, and thus we have the bound $Q\leq q\leq 2Q$ with $\log X\leq Q\leq X^{\delta}.$ Furthermore, we recall the hypothesis $X^{2d}\leq A$ in the statement in this lemma. Hence, it follows from $(\ref{4.21})$ that 
\begin{equation}\label{4.22}
    T(\alpha)\ll AX^2Q^{-2/d+\epsilon}.
\end{equation}
Therefore, on substituting ($\ref{4.22}$) into ($\ref{4.18}$), we obtain the bound
\begin{equation*}
     \dsum_{\substack{\|\a\|_{\infty}\leq A\\ P(\a)=0}}\left|\mathcal{I}_{\a}(X,\mathfrak{C})\right|^2\ll A^{N-n/8-k}X^{7n/4}\textrm{mes}(\mathfrak{C})^2(AX^2Q^{-2/d+\epsilon})^{n/8}.
\end{equation*}
On noting that $\textrm{mes}(\mathfrak{C})\ll Q^3(AX^d)^{-1}$, we conclude that
\begin{equation}\label{4.23}
      \dsum_{\substack{\|\a\|_{\infty}\leq A\\ P(\a)=0}}\left|\mathcal{I}_{\a}(X,\mathfrak{C})\right|^2\ll A^{N-k-2}X^{2n-2d}Q^{6-n/(4d)+\epsilon}.
\end{equation}
We find from ($\ref{4.23}$) together with the hypothesis $n=8s+r > 24d$ in this lemma that
$$  \dsum_{\substack{\|\a\|_{\infty}\leq A\\ P(\a)=0}}\left|\mathcal{I}_{\a}(X,\mathfrak{C})\right|^2\ll A^{N-k-2}X^{2n-2d}Q^{-1/(4d)+\epsilon}.$$
Hence, on recalling that $J_1=O(\log X)$ and $J_2=O(\log\log X)$, it follows from ($\ref{4.16}$) that
\begin{equation*}
    \dsum_{\substack{\|\a\|_{\infty}\leq A\\ P(\a)=0}}\left|\mathcal{I}_{\a}(X,\mathfrak{M}(X^{\delta})\setminus \mathfrak{M}(\log X))\right|^2\ll A^{N-k-2}X^{2n-2d}(\log X)^{-\eta},
\end{equation*}
for some $\eta>0.$
On noting that $A\leq X^{n-d}$ in the statement of this lemma, we complete the proof of Lemma $\ref{lemma 5.1}.$

\end{proof}

\section{Proof of Theorem 1.2}
In this section, we provide the proof of Theorem $\ref{thm2.2}$. To do this, we require two auxiliary lemmas  recorded in [\ref{ref10012}, Lemma 4.1 and 4.2].
In advance of the statement of the first lemma of these, it is convenient to define the exponential sum $S_{\a}(q)$ with $\a\in \Z^N$, $q\in \N$ by
\begin{equation*}
    S_{\a}(q):=S_{\a}(q;n)=q^{-n}\dsum_{\substack{1\leq b\leq q\\ (q,b)=1}}\dsum_{\substack{1\leq \r\leq q\\\r\in \Z^n}}e\left(\frac{b}{q}f_{\a}(\r)\right).
\end{equation*}
\begin{lem}\label{lem6.1}
Let $n$ and $d$ be natural numbers. Suppose that $A,B,C$ are sufficiently large positive numbers with $B<C.$ Suppose that $P\in \Z[\x]$ is a non-singular form in $N_{d,n}$ variables of degree $k\geq 2.$ Then, for any set $\mathcal{C}\subseteq [B,C]\cap \Z$, whenever $N\geq 200k(k-1)2^{k-1}$  we have 
\begin{equation*}
    \dsum_{\substack{\|\a\|_{\infty}\leq A\\ P(\a)=0}}\biggl|\dsum_{q\in \mathcal{C}}S_{\a}(q)\biggr|^2\ll A^{N-k}\biggl(\dsum_{q\in \mathcal{C}}q^{1+\epsilon}(q^{-1}+q^{-4/d})^{n/16}\biggr)^2.
\end{equation*}
\end{lem}
\begin{proof}
    See [\ref{ref10012}, Lemma 4.1].
\end{proof}

In advance of the statement of the second auxiliary lemma, we define 
\begin{equation}\label{defnJ}
   \mathfrak{J}_{\a}(w)=X^{n-d}A^{-1} \dint_{|\beta|\leq w}\dint_{[0,1]^n}e(\beta A^{-1}f_{\a}(\boldsymbol{\gamma}))d\boldsymbol{\gamma} d\beta,
\end{equation}
with $\a\in \Z^N.$
Furthermore, we recall the definition ($\ref{defnJ*}$) of $\mathfrak{J}_{\a}^*.$
\begin{lem}\label{lemma6.2}
Let $n$ and $d$ be natural numbers with $n\geq 8(d+1).$ Suppose that $P\in \Z[\x]$ is a non-singular form in $N_{d,n}$ variables of degree $k\geq 2.$ Then, whenever $N\geq 200k(k-1)2^{k-1}$, we have
    \begin{equation*}
        \dsum_{\substack{\|\a\|_{\infty}\leq A\\ P(\a)=0}}\left|\mathfrak{J}_{\a}^*-\mathfrak{J}_{\a}(w)\right|^2\ll A^{N-k-2}X^{2n-2d}(w^{2-n/(2(d+1))}+\zeta).
    \end{equation*}
\end{lem}
\begin{proof}
The only difference between Lemma \ref{lemma6.2} and [\ref{ref10012}, Lemma 4.2] is the choice of $\zeta.$ One readily sees that the choice of $\zeta=w^{-4-1/(8d)}$ does not harm the argument in [\ref{ref10012}, Lemma 4.2], and thus  [\ref{ref10012}, Lemma 4.2] still holds with $\zeta=w^{-4-1/(8d)}$.
\end{proof}

The proof of Theorem $\ref{thm2.2}$ follows the same strategy in the proof of [\ref{ref10012}, Theorem 1.1] with $\zeta=w^{-4-1/(8d)}$. To keep the paper self-contained, we restate the proof below under the condition $\zeta=w^{-4-1/(8d)}$.
\begin{proof}[Proof of Theorem 1.2]
Recall the definition $(\ref{I_a(X,B)})$ and $(\ref{4.24.24.2})$ of $\mathcal{I}_{\a}(X,\mathfrak{B})$ and $\mathfrak{M}$. Then, we find that
\begin{equation*}
\begin{aligned}
\mathcal{I}_{\a}(X,\mathfrak{M})&=\dint_{\mathfrak{M}}\dsum_{\substack{1\leq \x\leq X\\ \x\in \Z^{n}}}e(\alpha f_{\a}(\x)) d\alpha  \\
&=\dsum_{1\leq q\leq w}\dsum_{\substack{1\leq b\leq q\\(q,b)=1}}\dint_{|\alpha-b/q|\leq \frac{w}{AX^d}}\dsum_{\substack{1\leq \x\leq X}}e(\alpha f_{\a}(\x)) d\alpha.
\end{aligned} 
\end{equation*}
Recall the definition ($\ref{defnJ}$) of $\mathfrak{J}_{\a}(w).$ By applying classical treatments in major arcs [$\ref{ref8}$, Lemma 5.1] and writing $\beta=\alpha-b/q,$ we readily find that 
\begin{equation}\label{6.1}
\begin{aligned}
        &\mathcal{I}_{\a}(X,\mathfrak{M})=\dsum_{1\leq q\leq w}S_{\a}(q)\mathfrak{J}_a(w)+O(A^{-1}X^{n-d-1}w^5),
        \end{aligned}
\end{equation}
where 
$$S_{\a}(q)=\dsum_{\substack{1\leq b\leq q\\ (q,b)=1}}q^{-n}\dsum_{1\leq \r\leq q}e\left(\frac{b}{q} f_{\a}(\r)\right).$$
Meanwhile, recall the definition $(\ref{def2.3})$ and ($\ref{6.26.26.2}$) of $W$ and $\mathfrak{S}_{\a}^*$, and note from the classical treatment that 
$$\mathfrak{S}_{\a}^*=\displaystyle\prod_{p\leq w}\biggl(\dsum_{\substack{0\leq h\leq \log_pw}}S_{\a}(p^h)\biggr).$$
If we define a set 
\begin{equation*}
    \mathcal{Q}=\{q\in (w,W]:\ \text{for all primes}\ p,\ p^h\|q\Rightarrow p^h\leq w\}
\end{equation*}
and define $$\mathcal{E}_{\a}=\dsum_{q\in \mathcal{Q}}S_{\a}(q),$$ we find from the multiplicativity of $S_{\a}(q)$ that
\begin{equation}\label{6.2}
    \dsum_{1\leq q\leq w}S_{\a}(q)=\mathfrak{S}_{\a}^*-\mathcal{E}_{\a}.
\end{equation}

Meanwhile, note that $$\mathcal{I}_{\a}(X)=\mathcal{I}_{\a}(X,\mathfrak{M})+\mathcal{I}_{\a}(X,\mathfrak{m}).$$ Then, on recalling the definition of $\mathfrak{J}_{\a}^*$, we deduce from $(\ref{6.1})$ and $(\ref{6.2})$ together with applications of elementary inequality $(a+b)^2\leq 2a^2+2b^2$ that
\begin{equation}\label{6.3}
\begin{aligned}
      &\dsum_{\substack{\|\a\|_{\infty}\leq A\\P(\a)=0}}\left|\mathcal{I}_{\a}(X)-\mathfrak{S}_{\a}^*\mathfrak{J}_{\a}^*\right|^2\\
      &\ll \dsum_{\substack{\|\a\|_{\infty}\leq A\\P(\a)=0}}\left|\mathcal{I}_{\a}(X,\mathfrak{M})-\mathfrak{S}_{\a}^*\mathfrak{J}_{\a}^*\right|^2+\dsum_{\substack{\|\a\|_{\infty}\leq A\\P(\a)=0}}\left|\mathcal{I}_{\a}(X,\mathfrak{m})\right|^2\\
      &\ll \Sigma_1+\Sigma_2+\dsum_{\substack{\|\a\|_{\infty}\leq A\\P(\a)=0}}\left|\mathcal{I}_{\a}(X,\mathfrak{m})\right|^2+O\biggl(A^{-2}X^{2n-2d-2}w^{10}\dsum_{\substack{\|\a\|_{\infty}\leq A\\P(\a)=0}}1\biggr),
\end{aligned}
\end{equation}
where 
$$\Sigma_1=\dsum_{\substack{\|\a\|_{\infty}\leq A\\P(\a)=0}}\left|\mathcal{E}_{\a}\mathfrak{J}_{\a}^*\right|^2\ \text{and}\ \Sigma_2=\dsum_{\substack{\|\a\|_{\infty}\leq A\\P(\a)=0}}\left|(\mathfrak{S}_{\a}^*-\mathcal{E}_{\a})(\mathfrak{J}_{\a}^*-\mathfrak{J}_{\a}(w))\right|^2.$$

First, we estimate the third and fourth terms of the last expression in ($\ref{6.3}$). Since we have $$\#\{\a\in \Z^N:\ \|\a\|_{\infty}\leq A,\ P(\a)=0\}\ll A^{N-k},$$ it follows by Lemma $\ref{lem4.6}$ that 
\begin{equation}\label{6.161616}
\begin{aligned}
&\dsum_{\substack{\|\a\|_{\infty}\leq A\\P(\a)=0}}\left|\mathcal{I}_{\a}(X,\mathfrak{m})\right|^2+O\biggl(A^{-2}X^{2n-2d-2}w^{10}\dsum_{\substack{\|\a\|_{\infty}\leq A\\P(\a)=0}}1\biggr)\\
&\ll A^{N-k-2}X^{2n-2d}(\log A)^{-\delta},    
\end{aligned}
\end{equation}
with some $\delta>0.$ 

Next, we turn to estimate the first term of the last expression in ($\ref{6.3}$). From the trivial bound, we have 
\begin{equation}\label{6.1616}
\left|\mathfrak{J}_{\a}^*\right|^2\ll X^{2(n-d)}\zeta^{-2}A^{-2}=X^{2(n-d)}w^{8+1/(4d)}A^{-2}.
\end{equation}
By Lemma $\ref{lem6.1}$ with $k=2,$ $B=w,C=W$ and $\mathcal{C}=\mathcal{Q}$, we have
\begin{equation}\label{6.17}
    \begin{aligned}
        \dsum_{\substack{\|\a\|_{\infty}\leq A\\P(\a)=0}}\left|\mathcal{E}_{\a}\right|^2&\ll A^{N-k}\biggl(\dsum_{q\geq w}q^{1+\epsilon}(q^{-1}+q^{-4/d})^{n/16}\biggr)^2\\
        &\ll A^{N-k}\biggl(\dsum_{q\geq w}\left(q^{1-n/16+\epsilon}+q^{1-n/(4d)+\epsilon}\right)\biggr)^2.
    \end{aligned}
\end{equation}
Meanwhile, from the hypotheses $n=8s+r$ with $s\in \N$ and $1\leq r\leq 8,$ $s\geq 3d$ and $X^{2d}\leq A\leq X^{s-d}$ in the statement of Theorem $\ref{thm2.2},$ we see that $n>24d.$ Hence, it follows from $(\ref{6.17})$ together with the hypothesis $d\geq 4$ that
\begin{equation}\label{6.18}
\begin{aligned}
    \dsum_{\substack{\|\a\|_{\infty}\leq A\\P(\a)=0}}\left|\mathcal{E}_{\a}\right|^2&\ll  A^{N-k}\left(w^{2-3d/2}+w^{-4-1/(4d)}\right)^2\ll A^{N-k}\cdot w^{-8-1/(2d)}.
\end{aligned}
\end{equation}
Therefore, combining $(\ref{6.1616})$ and $(\ref{6.18})$, we conclude that
\begin{equation}\label{5.202020}
\Sigma_1=    \dsum_{\substack{\|\a\|_{\infty}\leq A\\P(\a)=0}}\left|\mathcal{E}_{\a}\mathfrak{J}_{\a}^*\right|^2\ll A^{N-k-2}X^{2n-2d}w^{-1/(4d)}.
\end{equation}

Lastly, it remains to estimate the second term of the last expression in ($\ref{6.3}$). From the trivial bound, we have
$$\left|\mathfrak{S}_{\a}^*-\mathcal{E}_{\a}\right|^2= \biggl|\dsum_{1\leq q\leq w}S_{\a}(q)\biggr|^2\leq w^4.$$
Hence, we deduce by applying Lemma $\ref{lemma6.2}$ with $n> 24d$ and $d\geq 4$ that
\begin{equation}\label{5.3030}
    \Sigma_2=\dsum_{\substack{\|\a\|_{\infty}\leq A\\P(\a)=0}}\left|(\mathfrak{S}_{\a}^*-\mathcal{E}_{\a})(\mathfrak{J}_{\a}^*-\mathfrak{J}_{\a}(w))\right|^2\ll A^{N-k-2}\cdot X^{2n-2d}\cdot w^{-1/(8d)}.
\end{equation}
 
 Then, on recalling the definition of $w$ and substituting $(\ref{6.161616})$, $(\ref{5.202020})$ and $(\ref{5.3030})$ into the last expression in ($\ref{6.3}$), one concludes that 
\begin{equation*}
    \dsum_{\substack{\|\a\|_{\infty}\leq A\\P(\a)=0}}\left|\mathcal{I}_{\a}(X)-\mathfrak{S}_{\a}^*\mathfrak{J}_{\a}^*\right|^2\ll A^{N-k-2}X^{2n-2d}(\log A)^{-\delta},
\end{equation*}
for some $\delta$ with $0<\delta<1.$ This completes the proof of Theorem $\ref{thm2.2}.$
\end{proof}


\end{document}